\documentclass[letterpaper, 12pt]{article}

\usepackage[T1]{fontenc}

\usepackage[margin=1in]{geometry}

\usepackage{amscd}

\usepackage{amsmath}

\usepackage{amsfonts}

\usepackage{amssymb}

\usepackage{amsthm}

\usepackage{arydshln}

\usepackage{fancyhdr}

\usepackage{verbatim}

\usepackage{tikz-cd}

\usepackage[all]{xy}

\usepackage{color}

\usepackage[colorlinks=true, linkcolor=blue, citecolor=blue,
pagebackref=true]{hyperref}

\usepackage{fouriernc}

\usepackage{mathtools} \usepackage{etoolbox}
\theoremstyle{definition}

\theoremstyle{plain} \newtheorem{theorem}{Theorem}[section]

\theoremstyle{plain} \newtheorem{lemma}[theorem]{Lemma}

\theoremstyle{plain} \newtheorem{proposition}[theorem]{Proposition}

\theoremstyle{plain} 

\theoremstyle{plain} \newtheorem{corollary}[theorem]{Corollary}

\theoremstyle{remark} \newtheorem*{remark}{Remark}

\theoremstyle{definition} 

\theoremstyle{definition} \newtheorem*{definition*}{Definition}

\theoremstyle{definition} \newtheorem{question}[theorem]{Question}

\theoremstyle{remark} 

\makeatletter \renewenvironment{proof}[1][\proofname]
{\par\pushQED{\qed}\normalfont\topsep6\p@\@plus6\p@\relax\trivlist\item[\hskip\labelsep\bfseries#1\@addpunct{.}]\ignorespaces}{\popQED\endtrivlist\@endpefalse}
\makeatother

\newcommand{\EE}{\mathbb{E}}

\newcommand{\PP}{\mathbb{P}}

\newcommand{\RR}{\mathbb{R}}

\newcommand{\QQ}{\mathbb{Q}}

\newcommand{\NN}{\mathbb{N}}

\newcommand{\TT}{\mathbb{T}}

\newcommand{\ZZ}{\mathbb{Z}}

\newcommand{\calA}{\mathcal{A}}

\newcommand{\calF}{\mathcal{F}}

\newcommand{\bone}{\mathbf{1}}

\renewcommand{\leq}{\leqslant} \renewcommand{\geq}{\geqslant}

\DeclarePairedDelimiter{\abs}{\lvert}{\rvert}
\DeclarePairedDelimiter{\norm}{\lVert}{\rVert}
\DeclarePairedDelimiter{\floor}{\lfloor}{\rfloor}
\DeclarePairedDelimiter{\set}{\lbrace}{\rbrace}
\DeclarePairedDelimiter{\parens}{\lparen}{\rparen}
\DeclarePairedDelimiter{\brackets}{\lbrack}{\rbrack}

\DeclareMathOperator{\Leb}{Leb}

\def\eps{{\varepsilon}}

\def\1int{{[0,1]}}

\title{Remarks about inhomogeneous pair correlations}

\author{Felipe A.~Ram{\'i}rez\footnote{\texttt{framirez@wesleyan.edu}}
  \\ Wesleyan University}

\date{}

\begin{document}

\maketitle

% ===============================================

{\centering\footnotesize \emph{For Jorge A.~Ram{\'i}rez (1954--2020)---\\with
    Love and Gratitude}\par}

\begin{abstract}
  Given an infinite subset $\calA \subseteq\NN$, let $A$ denote its
  smallest $N$ elements. There is a rich and growing literature on the
  question of whether for typical $\alpha\in[0,1]$, the pair
  correlations of the set $\alpha A \pmod 1\subset [0,1]$ are
  asymptotically Poissonian as $N$ increases. We define an
  inhomogeneous generalization of the concept of pair correlation, and
  we consider the corresponding doubly metric question. Many of the
  results from the usual setting carry over to this new
  setting. Moreover, the double metricity allows us to establish some
  new results whose singly metric analogues are missing from the
  literature.
\end{abstract}

% ===============================================

\thispagestyle{empty}

\setcounter{tocdepth}{1} %excludes subsections from toc
%\tableofcontents
{\footnotesize{\tableofcontents}}

\section{Metric Poissonian pair correlations}\label{sec:intr-pair-corr-1}

Given a sequence $\mathbf{x} = (x_n)_{n=1}^\infty$ of points on the
torus $\TT=\RR/\ZZ$, a point $\gamma\in\TT$, and a real number $s>0$,
we are interested in the asymptotic frequency with which
$x_i - x_j\, (i,j \leq N)$ lies in the arc
$[\gamma-s/N, \gamma + s/N] \subset \TT$. That is, we study the
limiting behavior of
\begin{equation*}
  F(\gamma, s, N, \mathbf{x}) = \frac{1}{N}\#\set*{(i, j)\in [N]\times[N]  \mid i\neq j, \norm{x_i - x_j - \gamma} \leq \frac{s}{N}},
\end{equation*}
where $\norm{\cdot}$ denotes distance to $0\in \TT$ and
$[N]:=\set{1, 2, \dots, N}$. For an increasing sequence of natural numbers
$\calA = (a_n)_{n=1}^\infty$ and
$\alpha, \gamma\in\TT$, we denote
$F (\alpha, \gamma, s, N, \calA)=F(\gamma, s, N, \mathbf x)$, where
the sequence $\mathbf x$ is defined by $x_n = a_n\alpha \pmod 1$.

Much attention is paid to the behavior of $F$ when $\gamma=0$. In this case 
$F(0,s, N, \mathbf x)$ is called the \emph{pair correlation statistic} of
$\mathbf x$. One says that $\mathbf x$ has \emph{Poissonian pair correlations (PPC)}
if
\begin{equation*}
(\forall s>0)\qquad \lim_{N\to\infty} F(0, s, N, \mathbf x)= 2s.
\end{equation*}
Like equidistribution, PPC is a marker of randomness; a sequence of
points on the circle which have been chosen independently and
uniformly at random will almost surely have Poissonian pair
correlations, just as they will almost surely be equidistributed in
the circle. In fact, PPC is a stronger feature of randomness than
equidistribution is, in the sense that any sequence which has PPC must
also be
equidistributed~\cite{AistleitnerLachmannPausinger,GrepstadLarcher,HKLSUPPC,
  Marklof2019}. The converse fails. For example, an orbit of any
irrational circle rotation equidistributes, but does not have
Poissonian pair correlations. Indeed, the three gaps theorem implies
that the gap distribution of the points of such an orbit is far from
random.

In the past two decades, one of the questions of greatest interest in
this area has been whether a given $\calA\subset\NN$ has \emph{metric
  Poissonian pair correlations (MPPC)}, that is, if for Lebesgue
almost every $\alpha\in\TT$, the sequence
$\parens*{a_n\alpha \pmod 1}_{n=1}^\infty$ has Poissonian pair
correlations. Inspired by a problem in quantum mechanics, Rudnick and
Sarnak~\cite{RudnickSarnak} showed that $\parens*{n^k}_{n\geq 1}$ has
MPPC whenever $k\geq 2$. But $\parens{n}_{n\geq 1}$ (that is,
$\calA = \NN$) does not have MPPC because for every $\alpha$ the
corresponding sequence on $\TT$ is an orbit of the circle rotation
over angle $2\pi\alpha$, and, as we have mentioned, orbits of circle
rotations do not have PPC. To put it informally, the problem for
$\calA = \NN$ arises from the fact that initial strings from the
sequence $\NN$ have too much additive structure; there are too many
different ways to achieve any given $d \in [N]-[N]$ as a difference of
two elements of $[N]$, and as a result, $F(\alpha, 0, s, N, \NN)$
counts events that have been rigged by this extra structure of $[N]$
to occur with non-random
regularity. In~\cite{Aistleitneretaladditiveenergy}, Aistleitner,
Larcher, and Lewko made an important forward stride in the study of
MPPC by putting this informal reasoning on a rigorous footing. They
connected the pair correlations of $\calA$ to the asymptotic behavior
of the additive energy
\begin{equation*}
  E(A) := \#\set*{(a,b,c,d)\in A^4 \mid a+b=c+d}
\end{equation*}
where $A:=A_N$ denotes the smallest $N$ elements of $\calA$. Specifically,
they proved the following theorem.

\begin{theorem}[{\cite[Theorem~1]{Aistleitneretaladditiveenergy}}]\label{thm:all}
  For an infinite subset $\calA \subset \NN$, let $A_N$ denote its
  smallest $N$ elements. If there exists some $\delta>0$ such that
  \begin{equation*}
    E(A_N) \leq N^{3-\delta}
  \end{equation*}
  for all sufficiently large $N$, then $\calA$ has MPPC.
\end{theorem}

\begin{remark}
  The Rudnick--Sarnak result follows from Theorem~\ref{thm:all}. On
  the other hand, it is easy to see that Theorem~\ref{thm:all} does
  not apply to $\calA=\NN$, since in this case one has
  $E(A_N)\gg N^3$.
\end{remark}

\begin{remark}[On notation]
  For functions $f,g:\NN\to\RR_{\geq 0}$ we use $f\ll g$ to mean
  $f= O(g)$ and $f\gg g$ to mean $g = O(f)$. If both hold, we write
  $f\asymp g$. By $f\sim g$ we mean that $(f/g)\to 1$ as the argument
  increases to $\infty$.
\end{remark}

Aistleitner, Larcher, and Lewko asked whether $E(A_N) = o(N^3)$ is
necessary and sufficient for $\calA$ to have MPPC. In an appendix to
their paper, Bourgain answered the question: it is necessary but not
sufficient. This left the question of whether there is some additive
energy threshold separating the sets that have MPPC from those that
do not. In~\cite{Bloometal}, Bloom, Chow, Gafni, and Walker formulated
the following question proposing a location for such a threshold:

\begin{question}[{\cite[Fundamental Question~1.7]{Bloometal}}]\label{q:fq}
  Let $\calA\subset\NN$ be an infinite set and suppose that
  \begin{equation*}
    E(A_N) \sim N^3\psi(N)
  \end{equation*}
  where $\psi: \NN \to [0,1]$ is some weakly decreasing function. Is
  convergence of the series $\sum\psi(N)/N$ necessary and sufficient for $\calA$
  to have metric Poissonian pair correlations?
\end{question} 

The answer is no, as we will soon see, but there was good reason to
believe otherwise. Previously, Walker had proved that the set of
primes does not have MPPC~\cite{Walker}. When $\calA$ is the set of
primes, one has $E(A_N)\asymp N^3(\log N)^{-1}$. Moreover, Bloom
\emph{et al.}~constructed sets with
\begin{equation*}
  E(A_N) \asymp \frac{N^3}{\log N \log\log N}
\end{equation*}
which also do not have MPPC~\cite{Bloometal}. (And later, Lachmann and
Technau constructed sets with
\begin{equation}\label{eq:12}
  E(A_N) \asymp \frac{N^3}{\log N \log\log N \dots \underbrace{\log\dots\log}_nN}
\end{equation}
which do not have MPPC~\cite{LachmannTechnau}.) Furthermore, one can
construct for any $\eps>0$ a set $\calA$ having
\begin{equation*}
  E(A_N) \asymp \frac{N^3}{\log N (\log\log N)^{1+\eps}},
\end{equation*}
such that $\calA$ does have
MPPC~\cite{AistleitnerLachmannTechnau}. Finally, since the methods of
proof in some of the above results involved Khintchine's theorem on
Diophantine approximation in a crucial way, it was natural to suspect
that the condition in Khintchine's theorem---\emph{i.e.}~convergence
of the sum $\sum\psi(n)/n$---is what should define the necessary and
sufficient threshold in the Question~\ref{q:fq}. 

Regarding necessity, Aistleitner, Lachmann, and Technau showed that
for every $\eps>0$ there exists $\calA$ having MPPC such that
\begin{equation*}
  E(A_N) \gg \frac{N^3}{(\log N)^{\frac{3}{4}+\eps}}
\end{equation*}
thereby answering that part of Question~\ref{q:fq} in the negative,
and putting an end to the idea that there is a threshold at
all~\cite{AistleitnerLachmannTechnau}---at least a \emph{two-way}
threshold. Sufficiency remains open. The main results of~\cite{Bloometal}
were in support of it. For example, Bloom \emph{et al.}~proved the
following.

\begin{theorem}[{\cite[Theorem~1.4]{Bloometal}}]
  If there is some constant $\xi > 0$ for which
  \begin{equation*}
    E(A_N) \ll \frac{N^3}{(\log N)^{2+\xi}} \quad\textrm{and}\quad \delta(A_N) \gg \frac{1}{(\log N)^{2+2\xi}},
  \end{equation*}
  where $\delta(A_N)$ denotes density, then $\calA$ has MPPC.
\end{theorem}

Later, Bloom and Walker added more evidence for the sufficiency part
of Question~\ref{q:fq} in the form of the following result.

\begin{theorem}[{\cite[Theorem~6]{BloomWalker}}]
  There exists an absolute positive constant $C$ such that for any
  $\calA\subset\NN$,
  \begin{equation}\label{eq:bloomwalker}
    E(A_N) \ll \frac{N^3}{(\log N)^C}
  \end{equation}
  implies that $\calA$ has MPPC.
\end{theorem}

Of course, Question~\ref{q:fq} supposes that that constant can be any
$C>1$. Indeed, for a version of the problem on the $d$-dimensional
torus (where $d\geq 2$), Hinrichs \emph{et al.}~have shown that if
there exists some $C>1$ for which~(\ref{eq:bloomwalker}) holds, then
$\calA$ has ``MPPC in dimension $d$''~\cite{HKLSUPPC}.

No further progress has been made on the sufficiency part of
Question~\ref{q:fq}. One of the results of this
note---Theorem~\ref{thm:fundyes}---establishes an inhomogeneous
analogue of the sufficiency part of the question. We introduce the
inhomogeneous problem in the next section.

\section{Doubly metric Poissonian pair correlations}
\label{sec:inhom-pair-corr}

%\subsection*{A doubly metric definition} 

As we have mentioned, an independent, identically distributed sequence
$\mathbf x=(x_n)_n$ of random variables having the uniform
distribution on $\TT$ will almost surely have Poissonian pair
correlations. In fact, it is not hard to show that for any fixed
$\gamma\in\TT$ such a random sequence $\mathbf x\subset\TT$ almost surely satisfies
\begin{equation}\label{eq:gammappc}
  (\forall s>0)\qquad \lim_{N \to \infty}F (\gamma, s, N,\mathbf x) =  2s.
\end{equation}
When~(\ref{eq:gammappc}) holds for a sequence $\mathbf x \in \TT$, we
will say it has \emph{Poissonian pair correlations with inhomogeneous
  parameter $\gamma\in\TT$ (or $\gamma$-PPC)}. Notice then that, by
Fubini's theorem, the randomly generated sequence $\mathbf x$ will
almost surely have $\gamma$-PPC for a full-measure set of
$\gamma\in\TT$. We will present a proof of the following proposition
that serves as template for some of the other results.

\begin{proposition}\label{prop:random}
  Let $\gamma$ be chosen randomly and uniformly from $\TT$ and
  $\mathbf x = (x_n)_{n=1}^\infty$ be a randomly, uniformly, and
  independently chosen sequence of points on $\TT$. Then, almost
  surely, $\mathbf x$ has $\gamma$-PPC. (That is,~(\ref{eq:gammappc})
  almost surely holds.)
\end{proposition}

With this proposition mind, we propose to study a doubly metric
inhomogeneous variant of the notion of metric Poissonian pair
correlations. For an infinite subset $\calA\subset\NN$ and an
increasing integer sequence $\set{N_t}$, we will say
that $\calA$ has \emph{doubly metric Poissonian pair correlations
  (DMPPC) along the subsequence $\set{N_t}\subset\NN$} if for almost
all pairs $(\alpha, \gamma)$, we have
\begin{equation}\label{eq:dmppc}
  (\forall s>0)\qquad \lim_{t\to \infty}F (\alpha, \gamma, s, N_t, \calA) =  2s.
\end{equation}
If this holds with $\set{N_t} = \NN$, then we will just say that
$\calA$ has DMPPC.

%\subsection*{Sufficiency in the Fundamental Question}

\

The double metricity of the inhomogeneous set up makes certain
calculations easier than in the homogeneous setting. As a result, we
are able to confirm the sufficiency part of Question~\ref{q:fq} for
DMPPC.

\begin{theorem}\label{thm:fundyes}
  If $\calA\subset\NN$ is an infinite set such that
  \begin{equation*}
    E(A_N) \ll N^3\psi(N)
  \end{equation*}
  where $\psi:\NN\to[0,1]$ is a weakly decreasing function such that
  the series $\sum \psi(N)/N$ converges, then $\calA$ has doubly
  metric Poissonian pair correlations. (That is,~(\ref{eq:dmppc}) holds
  for almost all pairs $(\alpha, \gamma)\in\TT^2$.)
\end{theorem}

\begin{remark}
  In particular, if there is some $C>1$ for
  which~\eqref{eq:bloomwalker} holds, then $\calA$ has DMPPC (a result
  which in the homogenous setting is only known for higher
  dimensions).
\end{remark}

For functions on the divergence side of Question~\ref{q:fq},  we
construct sets $\calA\subset\NN$ that do not have DMPPC. 

\begin{theorem}\label{thm:div}
  Suppose $\psi:\NN\to[0,1]$ is a weakly decreasing function such that
  there exists $\delta>0$ for which $N^{3-\delta}\psi(N)$ is
  increasing, and such that $\sum\psi(N)/N$ diverges. Then there
  exists an infinite set $\calA\subset\NN$ such that
  \begin{equation*}
    E(A_N) \asymp N^3\psi(N)
  \end{equation*}
  and such that $\calA$ does not have doubly metric Poissonian pair correlations.
\end{theorem}

\begin{remark}
  The assumption that $N^{3-\delta}\psi(N)$ is increasing for some
  $\delta>0$ is only used to ensure that the sets we construct
  actually satisfy $E(A_N)\asymp N^3\psi(N)$. (Notice that since
  $E(A_N)$ increases to infinity, it is natural that there ought to be extra
  requirements on $\psi$ besides just divergence of $\sum\psi(N)/N$.) 
  Theorem~\ref{thm:div} implies, in particular, that there are sets
  $\calA\subset\NN$ satisfying~\eqref{eq:12} that do not have DMPPC. 
\end{remark}

We leave open the full necessity part of Question~\ref{q:fq} for
doubly metric Poissonian pair correlations. It is not entirely clear
what the answer should be. The construction of
Aistleitner--Lachmann--Technau~\cite{AistleitnerLachmannTechnau}
depends crucially on aspects of homogeneous Diophantine approximation
(like continued fractions), so it is not immediately obvious whether
a similar strategy would work for the doubly metric problem.

\

Some of the challenges of establishing DMPPC vanish if we allow
ourselves to consider the limit in~\eqref{eq:dmppc} along subsequences
$\set{N_t}\subseteq \NN$. For DMPPC along sequences, there is the
following.

\begin{theorem}\label{thm:subsequenceofsubsequence}
  Let $\calA\subset\NN$ be an infinite set. For any subsequence
  $\set{N_t}\subset\NN$ such that
  \begin{equation*}
    \lim_{t\to\infty} N_t^{-3}E(A_{N_t})=0,
  \end{equation*}
  there exists a subsequence of $\set{N_t}$ along which $\calA$ has
  doubly metric Poissonian pair correlations.
\end{theorem}

In particular, Theorem~\ref{thm:subsequenceofsubsequence} implies that
if $E(A_N)=o(N^3)$, then any integer sequence has a
subsequence along which $\calA$ has DMPPC. For example, since we have
$E(A_N) \asymp N^3(\log N)^{-1}$ when $\calA$ is the set of prime
numbers, Theorem~\ref{thm:subsequenceofsubsequence} leads immediately
to the following.

\begin{corollary}\label{cor:primes}
  Every increasing integer sequence has a subsequence
  along which the primes have doubly metric Poissonian pair correlations.
\end{corollary}

Combining Theorem~\ref{thm:subsequenceofsubsequence} with Bourgain's
argument in~\cite{Aistleitneretaladditiveenergy} leads to the
following necessary and sufficient condition for the existence of a
sequence along which $\calA$ has DMPPC.

\begin{theorem}\label{thm:subsequence}
  For an infinite set $\calA\subset\NN$, there exists an integer
  sequence along which $\calA$ has doubly metric Poissonian pair
  correlations if and only if $\liminf_{N\to\infty} N^{-3}E(A_N)=0.$
\end{theorem}

It would be interesting to know whether
Theorems~\ref{thm:subsequenceofsubsequence} and~\ref{thm:subsequence}
also hold in the original homogeneous setting. For example, recall
that Walker proved that the primes do not have
MPPC~\cite{Walker}. Might it be the case that every increasing integer
sequence has a subsequence along which~\eqref{eq:dmppc} holds for
almost every $\alpha \in \TT$ and with $\gamma=0$?  That is, does
Corollary~\ref{cor:primes} hold for MPPC? Or should we take
Corollary~\ref{cor:primes} as evidence that the primes actually have
DMPPC?

%\subsection*{Inhomogeneous pair correlations and equidistribution}

\

Finally, It is known that having Poissonian pair correlations is a
stronger condition than being
equidistributed~\cite{AistleitnerLachmannPausinger, GrepstadLarcher,
  HKLSUPPC, Marklof2019}. We show that the same is true
inhomogeneously.

\begin{theorem}\label{eq:uid}
  If $\mathbf x$ is a sequence of points in $\TT$ and there exists
  some $\gamma$ for which $\mathbf x$ has Poissonian pair correlations
  with inhomogeneous parameter $\gamma$, then $\mathbf x$ is
  equidistributed in $\TT$.
\end{theorem}

%\subsection*{Other questions}

Before moving on to the proofs, we mention a number questions and
speculations that are not pursued here.

What is the relationship between MPPC and DMPPC? Does one imply the
other? From the point of view of our results, the most optimistic hope
would be for DMPPC to imply MPPC, because then
Theorem~\ref{thm:fundyes} would imply a positive answer to the
sufficiencty part of Question~\ref{q:fq} for MPPC. 

Another question which has been asked in the homogeneous setting and
can as easily be asked in the doubly metric case is that of a zero-one
law: If $\calA\subset\NN$ does not have DMPPC, is it the case
that~\eqref{eq:dmppc} fails for \emph{almost every} pair
$(\alpha, \gamma)$? For MPPC, the question is still open, with some
progress in the work of Lachmann--Technau~\cite{LachmannTechnau} and
Larcher--Stockinger~\cite{LarcherStockinger,LarcherStockingerMaxE,LarcherStockingerNegResults}. In
fact, Larcher and Stockinger have conjectured that even more is true;
namely, when $\calA\subset\NN$ does not have MPPC, then there is
\emph{no} $(\alpha, 0)$ for which~\eqref{eq:dmppc} holds. The doubly
metric version of this conjecture would immediately show that MPPC
implies DMPPC. However, that doubly metric statement seems unlikely,
and can perhaps be disproved by examining the examples
in~\cite{AistleitnerLachmannTechnau} and showing that they do not have
DMPPC. A more likely speculation would be that for every fixed
$\gamma$, if~\eqref{eq:dmppc} does not hold for almost every $\alpha$,
then it holds for no $\alpha$.

\section{Proof of Proposition~\ref{prop:random}}\label{sec:heuristics}

First, we prove Proposition~\ref{prop:random}. The result itself is
not surprising, but we include it anyway because it motivates the
definition of DMPPC and because its proof is a template for the later
proofs, particularly Theorem~\ref{thm:fundyes}.

Note that a homogeneous version of this argument would give the
corresponding result for (homogeneous) Poissonian pair
correlations---originally established
in~\cite{AistleitnerLachmannPausinger,MarklofDistmodone}.

\begin{proof}[Proof of Proposition~\ref{prop:random}]
  We will show that for any fixed $s>0$, we almost surely get
  \begin{equation*}
    \lim_{N\to\infty} \frac{1}{N}\#\set*{1\leq i\neq j \leq N : \norm*{x_i - x_j - \gamma} \leq \frac{s}{N}} = 2s. 
  \end{equation*}
  Since $s>0$ is arbitrary, this must almost surely hold simultaneously for a
  countable dense set of values for $s>0$. An approximation argument
  will do the rest.

  Let $s>0$. For each $i\neq j$ and $N$, let $\bone_{i,j,s/N}$ denote the indicator
  random variable for
  \begin{equation*}
    \set*{(\gamma, \mathbf x) : \norm{x_i-x_j-\gamma}\leq \frac{s}{N}}.
  \end{equation*}
  Then we can see the pair correlation function as the random variable
  \begin{equation}\label{eq:11}
    F(s,N) = \frac{1}{N}\sum_{1\leq i\neq j\leq N} \bone_{i,j,s/N}.
  \end{equation}
  Our goal is to show that $\PP(F(s, N)\to 2s) =1$.

  We now show that if $(i,j)\neq (k,\ell)$ then $\bone_{i,j,\eps}$ and
  $\bone_{k,\ell, \eps'}$ are uncorrelated (hence independent since they
  are indicators). Indeed,
  \begin{align*}
    \EE\parens*{\bone_{i,j,\eps}\bone_{k,\ell,\eps'}} &= \int_\TT
                                                     \int_{\TT^4} \parens*{\sum_{n\in\ZZ}c(n) e(n(x_i - x_j -
                                                     \gamma))}\parens*{\sum_{n\in\ZZ}c'(n) e(n(x_k - x_\ell -
                                                     \gamma))}\,d(x_i, x_j,x_k,x_\ell)\,d\gamma,
                                                     \intertext{  where $c(n)$ and $c'(n)$ are the Fourier coefficients of $\bone_{[-eps,\eps] +
                                                     \ZZ}$ and
                                                     $\bone_{[-\eps',\eps']+\ZZ}$,
                                                     respectively. Continuing, }
    &= \sum_{m,n\in\ZZ}c(m) c'(n) \int_\TT
      \int_{\TT^4}  e(m(x_i - x_j -
      \gamma) + n (x_k-x_\ell-\gamma)
      \,d(x_i, x_j,x_k,x_\ell)\,d\gamma.
  \end{align*}
  The integral over $\gamma$ separates, and is only nonzero when
  $m=-n$. So we have
  \begin{equation*}
    \EE\parens*{\bone_{i,j,\eps}\bone_{k,\ell,\eps'}}=
    \sum_{n\in\ZZ}c(n) c'(-n) \int_{\TT^4}  e(n(x_i - x_j - x_k+ x_\ell))
                                                       \,d(x_i, x_j,x_k,x_\ell)
  \end{equation*}
  Since $(i,j)\neq (k,\ell)$, the integral is $0$ unless $n=0$, so we
  are left with
  \begin{equation*}
    \EE\parens*{\bone_{i,j,\eps}\bone_{k,\ell,\eps'}}= c(0) c'(0) =\EE\parens*{\bone_{i,j,\eps}}\EE\parens*{\bone_{k,\ell,\eps'}},
  \end{equation*}
  which is what we claimed.

  Notice that $\EE(\bone_{i,j,s/N}) = 2s/N$. Therefore,
  by~\eqref{eq:11}, we have
  \begin{equation*}
    \EE(F(s,N)) = \frac{1}{N}\sum_{1\leq i\neq j \leq N}\frac{2s}{N} =
    \parens*{1-\frac{1}{N}}2s,
  \end{equation*}
  and
  \begin{align}
    \sigma^2(F(s,N)) &= \frac{1}{N^2}\sum_{1\leq i\neq j \leq
                       N}\sigma^2\parens*{\bone_{i,j,s/N}} \nonumber \\
    &= \frac{1}{N^2}\sum_{1\leq i\neq j \leq N}\brackets*{\frac{2s}{N}
      - \parens*{\frac{2s}{N}}^2} \ll \frac{s}{N}. \label{eq:1/N}
  \end{align}
  In the rest of the proof we will use these last findings to show that
  the probability that $F(s,N)$ is far from its expected value is small, and that the probability
  that that happens infinitely often is $0$.

  Let $\set{s_M}_{M=1}^\infty$ be any sequence converging to
  $s$. Notice that, by~(\ref{eq:1/N}), we can write
  \begin{equation*}
    \sigma^2(F(s_M, M^2)) \ll M^{-2}
  \end{equation*}
  for all sufficiently large $M$, with an implied constant which may
  depend on $s$, but is independent of the sequence $\set{s_M}$, since
  its terms are eventually bounded uniformly away from $\infty$.
  Then, by Chebyshev's inequality,
  \begin{equation*}
    \PP\brackets*{\abs*{F(s_M,M^2) - \parens*{1 - \frac{1}{M^2}}2s_M} > \frac{1}{M^{1/4}}} < \frac{\sigma^2(F(s_M,M^2))}{1/M^{1/2}} \ll \frac{1}{M^{3/2}}
  \end{equation*}
  holds for all large $M\in\NN$. Since the sum
  $\sum M^{-3/2}$ converges, the
  Borel--Cantelli Lemma says that we almost surely have
  \begin{equation*}
    \abs*{F(\gamma, s_M,M^2, \mathbf x) - \parens*{1 - \frac{1}{M^2}}2s_M} \leq \frac{1}{M^{1/4}}
  \end{equation*}
  for all sufficiently large $M$, which implies that we almost surely
  have $F(\gamma, s_M,M^2, \mathbf x) \to 2s$ as $M\to\infty$. Now, notice that for any
  integer $N$ such that $M^2 \leq N \leq (M+1)^2$, we have
  \begin{equation*}
    \frac{M^2}{(M+1)^2} F\parens*{\gamma, s\frac{M^2}{(M+1)^2}, M^2, \mathbf x} \leq F(\gamma, s, N, \mathbf x) \leq \frac{(M+1)^2}{M^2} F\parens*{\gamma, \frac{(M+1)^2}{M^2}s, (M+1)^2, \mathbf x}.
  \end{equation*}
  The left-most and right-most members almost surely converge to $2s$ as $M$
  increases, therefore, almost surely, $F(\gamma, s,N, \mathbf x) \to 2s$, as we wanted.
\end{proof}

\section{Proof of Theorem~\ref{thm:fundyes}}
\label{sec:proof-theor-refthm:f}

  For $\calA\subset\NN$, $s>0$, $N\in\NN$, let us regard $F(s,N,
  \calA)$ as a measurable function on the  space $\TT^2$ equipped with
  Lebesgue measure. Specifically, it is the function
  \begin{equation*}
    F(s,N,\calA) = \frac{1}{N} \sum_{\substack{(a,b)\in A^2\\ a\neq
        b}} \bone_{(a-b),s/N},
  \end{equation*}
  where, for $d\in \ZZ$ and $\eps>0$, we use $\bone_{d,\eps}$ to
  denote the indicator of the set
  \begin{equation*}
    \set*{(\alpha, \gamma)\in\TT^2 :
      \norm{d\alpha-\gamma} \leq \eps}. 
  \end{equation*}
  Given $d\in \ZZ$, let
  \begin{equation*}
    r_A(d) = \#\set*{(a,b)\in A^2 : a-b = d}
  \end{equation*}
  be the number of ways to represent $d$ as a difference of two
  elements of $A$. In particular, $r(d)$ is non-zero only if
  $d\in A-A$. It is a simple exercise to verify the following:
  \begin{align}
    \sum_{d\in\ZZ} r_A(d) &= N^2 \label{eq:8} \\
    \sum_{d\in\ZZ} r_A(d)^2 &= E(A) \label{eq:9} \\
    F(s,N,\calA) &= \frac{1}{N} \sum_{d\in\ZZ\setminus \set{0}} r_A(d) \bone_{d,s/N}. \label{eq:5}
  \end{align}
  The next lemma shows that the functions $\bone_{d,s/N}$ are
  pairwise independent as random variables on $\TT^2$.

  \begin{lemma}\label{lem:pairwise}
    For any sequence $(\eps_d)_{d\in\ZZ}$ of positive reals, the
    associated random variables $(\bone_{d,\eps_d})_{d\in\ZZ}$ are pairwise
    independent. 
  \end{lemma}

  \begin{proof}
    Let $\bone_{d_1, \eps_1}, \bone_{d_2, \eps_2}$ be any pair of
    distinct random variables from the sequence. We must show that
    \begin{equation}\label{eq:4}
      \int_{\TT^2} \bone_{d_1, \eps_1}\bone_{d_2, \eps_2} =
      \int_{\TT^2}\bone_{d_1, \eps_1}\int_{\TT^2}\bone_{d_2, \eps_2}. 
    \end{equation}
    We rewrite the left-hand side as
        \begin{align}
          \int_{\TT^2} \bone_{d_1, \eps_1} \bone_{d_2, \eps_2} &=
                                                          \int_{\TT^2}
                                                          \bone_{(-\eps_1, \eps_1) + \ZZ} (d\alpha - \gamma) \bone_{(-\eps_2, \eps_2) + \ZZ} (d\alpha - \gamma)\,d\alpha\,d\gamma
          \nonumber \\
                                                        &= \int_{\TT^2}
                                                          \parens*{
                                                          \sum_{n_1\in\ZZ}c_1(n_1)e(n_1(d_1\alpha
                                                          -
                                                          \gamma))}\parens*{
                                                          \sum_{n_2\in\ZZ}c_2(n_2)e(n_2(d_2\alpha
                                                          -
                                                          \gamma))}\,d\alpha\,d\gamma,\label{eq:3}
        \end{align}
        where $(c_i(n))_{n\in\ZZ}$ are the Fourier coefficients of
        $\bone_{(-\eps_i, \eps_i) + \ZZ}$. We may rewrite~\eqref{eq:3}
        as
        \begin{equation*}
          \sum_{(n_1,n_2) \in
            \ZZ^2} c_1(n_1)c_2(n_2) \int_\TT e\parens*{(d_1n_1 + d_2 n_2) \alpha}\,d\alpha\underbrace{\int_\TT e\parens*{- (n_1 + n_2)\gamma}\,d\gamma}.
        \end{equation*}
        The indicated integral is only nonzero if $n_1 + n_2 = 0$, so
        the expression becomes
        \begin{equation*}
          \sum_{n \in
            \ZZ} c_1(n)c_2(-n) \underbrace{\int_\TT e\parens*{n(d_1 -
              d_2) \alpha}\,d\alpha}.
        \end{equation*}
        Since $d_1\neq d_2$, the newly indicated integral is nonzero
        only if $n=0$, in which case it is $1$, so the expression
        becomes $c_1(0)c_2(0)$, which is exactly the right-hand side
        of~\eqref{eq:4}. 
  \end{proof}

  \begin{lemma}
  For any infinite set $\calA\subset \NN$, and for every $s>0$ and sufficiently large $N$, we have
  \begin{equation*}
    \int_{\TT^2} F(s, N, \calA) = \frac{2(N-1)}{N}s.
  \end{equation*}
\end{lemma}

\begin{proof}
  As we have seen in~\eqref{eq:5}, $F(s,N,\calA)$ is a linear
  combination of random variables of the form $\bone_{d,s/N}$ where
  only $d$ varies. Notice that
  \begin{equation*}
    \int_{\TT^2}\bone_{d,s/N} = \frac{2s}{N}
  \end{equation*}
  as long as $(2s)/N \leq 1$. Therefore, by~\eqref{eq:5}, we have
  \begin{equation*}
    \int F(s,N,\calA) = \frac{2s}{N^2} \sum_{d\in\ZZ\setminus\set{0}} r_A(d) \overset{\eqref{eq:8}}{=}
    \frac{2s}{N^2}(N^2-N) = \frac{2(N-1)}{N}s,
  \end{equation*}
  as claimed.
\end{proof}

\begin{lemma}\label{lem:variance}
  For any infinite subset $\calA\subset\NN$,  $N\in\NN$, and $s>0$, we have
  \begin{equation*}
    \sigma^2(F(s, N, \calA)) \ll E(A) N^{-3} s,
  \end{equation*}
  where $\sigma^2$ denotes variance.
\end{lemma}

\begin{proof}
We have found in Lemma~\ref{lem:pairwise} that the random variables
$\bone_{d,s/N}$ are pairwise independent. Therefore,
\begin{equation*}
  \sigma^2(F(s, N, \calA)) = \frac{1}{N^2} \sum_{d\in\ZZ\setminus\set{0}} r_A(d)^2 \sigma^2(\bone_{d,s/N}).
\end{equation*}
(The reader should keep in mind that this is really a finite sum,
since $r(d)$ is nonzero for only finitely many $d$.) Now, 
\begin{equation*}
  \sigma^2(\bone_{d,s/N}) = \int \bone_{d,s/N}^2- \parens*{\int
    \bone_{d,s/N}}^2 =
  \begin{cases}
    \frac{2s}{N} - \parens*{\frac{2s}{N}}^2 &\textrm{if } N\geq 2s \\
    0 &\textrm{if } N \leq 2s. 
  \end{cases}
\end{equation*}
Combining, we find
\begin{equation*}
  \sigma^2(F(s,N,\calA)) \leq \frac{2s}{N^3}
  \sum_{d\in\ZZ}r(d)^2 \overset{\eqref{eq:9}}{=} 2E(A) N^{-3}s,
\end{equation*}
which proves the lemma.
\end{proof}

Now we can state the proof of Theorem~\ref{thm:fundyes}. 

\begin{proof}[Proof of Theorem~\ref{thm:fundyes}]
  Let $\eps, s >0$. We have $E(A_N) \leq N^3\psi(N)$ where
  $\psi:\NN\to [0,1]$ is a weakly decreasing function such that
  $\sum \psi(N)/N$ converges. Therefore, for any fixed real number
  $k>1$, $\sum\psi(\floor{k^t})$ converges. In particular, by
  Lemma~\ref{lem:variance}, $\sum\sigma^2(F(s,\floor{k^t}, \calA))$
  converges. Meanwhile, Chebyshev's inequality says that the measure
  of the set of $(\alpha, \gamma)\in\TT$ for which
  \begin{equation*}
    \abs*{F(\alpha, \gamma, s, \floor{k^t}, \calA) - 2s\parens*{\frac{\floor{k^t}-1}{\floor{k^t}}}} \geq \eps
  \end{equation*}
  is no more than $\eps^{-2}\sigma^2(F(s,\floor{k^t}, \calA))$. Since
  $\sum\eps^{-2}\sigma^2(F(s,\floor{k^t}, \calA))$ converges, the Borel--Cantelli
  lemma tells us that for almost every $(\alpha, \gamma)$, there are at
  most finitely many $t$ for which
  \begin{equation*}
    \abs*{F(\alpha, \gamma, s, \floor{k^t}, \calA) - 2s\parens*{\frac{\floor{k^t}-1}{\floor{k^t}}}} < \eps
  \end{equation*}
  does not hold. Decreasing $\eps$ to $0$ along a positive real sequence, we
  conclude that for almost every $(\alpha, \gamma)$, 
  \begin{equation*}
    F(\alpha, \gamma, s, \floor{k^t}, \calA) \to 2s
  \end{equation*}
  as $t\to\infty$. 

  By repeating the argument of the previous paragraph for every rational
  multiple of $s$, we may say that for almost every $(\alpha, \gamma)$,
  the following holds:
  \begin{equation}\label{eq:1}
    (\forall r\in\QQ_{>0})\qquad F(\alpha, \gamma, r s,
    \floor{k^t},\calA) \to 2 r s \quad (t\to\infty).
  \end{equation}
  A review of the definitions reveals that if $k^t < N \leq k^{t+1}$, then
  \begin{equation*}
    \frac{\floor{k^t}}{N} F\parens*{\alpha, \gamma, s \frac{\floor{k^t}}{N}, \floor{k^t}} \leq F(\alpha,
    \gamma, s, N) \leq \frac{\floor{k^{t+1}}}{N} F\parens*{\alpha, \gamma, s \frac{\floor{k^{t+1}}}{N}, \floor{k^{t+1}}}.
  \end{equation*}
  From this, we find that
  \begin{equation*}
    \frac{1}{k}\parens*{1 - \frac{1}{k^t}} F\parens*{\alpha, \gamma, s/k \parens*{1 - \frac{1}{k^t}}, \floor{k^t}} \leq F(\alpha,
    \gamma, s, N) \leq  k F\parens*{\alpha, \gamma, sk ,
      \floor{k^{t+1}}}. 
  \end{equation*}
  Now, let $\delta>0$ be a rational number. For $t$ sufficiently
  large, we have
  \begin{equation*}
    1-\delta < \parens*{1 - \frac{1}{k^t}}  < 1, 
  \end{equation*}
  hence
  \begin{equation*}
    \frac{1}{k}\parens*{1 - \delta} F\parens*{\alpha, \gamma, (s/k) \parens*{1 - \delta}, \floor{k^t}} \leq F(\alpha,
    \gamma, s, N) \leq  k F\parens*{\alpha, \gamma, sk ,
      \floor{k^{t+1}}}
  \end{equation*}
  holds if $N$ is large enough.  Now if $k>1$ is rational, we may
  apply~\eqref{eq:1} to show that for almost every $(\alpha, \gamma)$,
  \begin{equation*}
    \frac{2s}{k^2} (1-\delta)^3 \leq F(\alpha, \gamma, s, N) \leq 2s k^2 (1+\delta)
  \end{equation*}
  for all sufficiently large $N$.  By taking $k \downarrow 1$ and
  $\delta \downarrow 0$ along sequences of rational numbers, and intersecting the
  corresponding full-measure sets of $(\alpha,\gamma)$, we see that
  for almost all $(\alpha, \gamma)$,
  \begin{equation}\label{eq:above}
    F(\alpha, \gamma, s, N)\to 2s, \quad N\to \infty. 
  \end{equation}
  Finally, by intersecting countably many full-measure subsets of
  $\TT^2$, we can say that for almost all $(\alpha, \gamma)$, and all
  rational $s>0$, the limit~(\ref{eq:above}) holds. This is enough to
  conclude that it holds for all $s$.
\end{proof}

\section{Proof of Theorem~\ref{thm:div}}
\label{sec:proof-diverg-part}

The proof of Theorem~\ref{thm:div} is based on Bourgain's construction
in~\cite{Aistleitneretaladditiveenergy}. We require a couple of
lemmas. The first lemma is a simple consequence of the fact that
circle expanding maps are mixing. It will allow us to define a
sequence $(U_N)_N$ of sets which are quasi-independent, meaning that
there is some constant $C>1$ such that $\Leb(U_M\cap U_N)\leq
C\Leb(U_M)\Leb(U_N)$ whenever $M\neq N$. 

\begin{lemma}\label{lem:deltas} 
For any sequence  of measurable sets $Q_N\in \TT$  it is possible to choose a sequence
of positive integers $\Delta_N$ so that the sets
  \begin{equation*}
    R_N = \set*{\alpha\in\TT : \Delta_N\alpha \in Q_N}
  \end{equation*}
  are pairwise quasi-independent.
\end{lemma}

\begin{proof}
Let $m\geq 2$ be an integer. Recall that the map $f_m: \TT\to\TT$ defined
by $f_m(\alpha) = m\alpha (\bmod 1)$ is measure-preserving and
mixing, meaning that for any measurable sets $S, T\subset\TT$ we have
that $\Leb(f_m^{-1}(S)) = \Leb(S)$ and also that
\begin{equation*}
  \lim_{k\to\infty} \Leb(f_m^{-k}(S)\cap T) = \Leb(S)\Leb(T). 
\end{equation*}
Notice that $R_N = f_{\Delta_N}^{-1}(Q_N)$.

We may put $\Delta_1=1$, and for each $N$, choose $\Delta_N$ to be a
large enough power $m^{k_N}$ that
\begin{equation*}
  \Leb(f_{\Delta_N}^{-1}(Q_N)\cap R_M) =
  \Leb(f_{m}^{-k_N}(Q_N)\cap R_M) \leq 2\Leb(Q_N)\Leb(R_M) = 2\Leb(R_N)\Leb(R_M)
\end{equation*}
holds for all $M=1, \dots, N-1$. 
\end{proof}

The next lemma is a measure estimate on the set of $(\alpha,\gamma)$
which simultaneously satisfy an inhomogeneous approximation condition
and a homogeneous approximation condition on $\alpha$.

\begin{lemma}\label{lem:farey}
  For any $M\in\NN$ and $0 \leq \sigma,\tau \leq 1/2$, let
  \begin{equation*}
    S = \set*{(\alpha,\gamma)\in\TT^2 : \abs*{\alpha - \frac{a}{d}}\leq \frac{\sigma}{M^2} \quad\textrm{for
        some}\quad d\in[M]\quad\textrm{and}\quad (a,d)=1},
  \end{equation*}
  and 
  \begin{equation*}
    T = \set*{(\alpha, \gamma) \in\TT^2: \norm{d\alpha -
        \gamma}\leq \frac{\tau}{M}\quad\textrm{for some}\quad d\in[M]}.
  \end{equation*}
  Then $\Leb(S\cap T) \gg \sigma\tau$. 
\end{lemma}

\begin{proof}
  Notice that $S$ consists of disjoint vertical strips in $\TT^2$ over
  the Farey fractions
  \begin{equation*}
    \calF_{M} := \set*{\frac{a}{d}\in [0,1]: (a,d)= 1}.
  \end{equation*}
  The strip $S_{a/d}$ over $a/d\in\calF$ has width
  $\frac{2\sigma}{M^2}$. Meanwhile, $T$ is a union of strips which
  wind around the torus, of slopes $1, \dots, M$ and with vertical
  cross-sections of length $2\tau/M$. Specifically, they are the
  supports of $\bone_{1,\tau}, \dots, \bone_{M,\tau}$. The condition
  $0< \sigma\leq 1/2$ guarantees that the indicators
  $\bone_{1,\tau}, \dots, \bone_{d,\tau}$, when restricted to the
  strip $S_{a/d}$, are mutually singular. They are supported on
  non-overlapping parallelograms contained in $S_{a/d}$, each of which
  has area
  $\parens*{\frac{2\sigma}{M^2}}\parens*{\frac{2\tau}{M}}$. Therefore,
  $\Leb(S_{a/d}\cap T) \geq d\frac{4\sigma\tau}{M^3}$. Summing over $a/d\in\calF_M$, we have
  \begin{equation}\label{eq:10}
    \Leb(S\cap T) \geq \sum_{d=1}^M \frac{4\sigma\tau}{M^3}d\varphi(d)
  \end{equation}
  where $\varphi(d)$ denotes the Euler totient function. The growth
  properties of $\varphi$ guarantee that $\sum_{d=1}^M d\varphi(d) \gg
  M^3$, and combining this with~\eqref{eq:10} proves the lemma.
\end{proof}

We now state the following.

\begin{proof}[Proof of Theorem~\ref{thm:div}]
  We adapt Bourgain's arguments
  from~\cite[Appendix]{Aistleitneretaladditiveenergy}. First, it is
  possible to modify them to show that if $\psi(N)\neq o(1)$, then no $\calA\subset\NN$ satisfying
  $E(A_N) \asymp N^3\psi(N)$ has DMPPC. We do this here in
  Theorem~\ref{thm:bourgain}. Therefore, let us assume that
  $\psi(N) = o(1)$.

  Next, notice that we may assume that $\psi(N)^{-1}\in\NN$ for every
  $N$, for example by replacing $\psi(N)^{-1}$ with
  $\floor{\psi(N)^{-1}}$.

  Let $\eps>0$ be a (small) constant, which we will specify later. Per
  Lemma~\ref{lem:deltas}, for each $N\in\NN$ let $\Delta_N$ be an
  integer large enough that the sets
  \begin{equation*}
    R_N = \set*{ \alpha\in\TT : \norm{d\Delta_N\alpha}\leq
      \frac{\psi(N)\eps}{N} \quad\textrm{for
        some}\quad 0 < d\leq
      N\eps}
  \end{equation*}
  are pairwise quasi-independent. For each $N$, set
  $S_N = R_N\times\TT\subset\TT^2$. Let
  \begin{equation*}
    T_N = \set*{(\alpha, \gamma)\in\TT^2 : \norm{d\Delta_N\alpha -
        \gamma} \leq \frac{1}{8N}\quad\textrm{for some}\quad 0 < d
      \leq \frac{N}{20 \psi(N)}}.
  \end{equation*}
  Observe that for all large $N$ the set $T_N$ contains the set
  $\Delta_N^{-1}T$, where $T$ is from Lemma~\ref{lem:farey}, with
  $\tau = \eps/8$, and that $S_N$ contains $\Delta_N^{-1}S$ with
  $\sigma = \psi(N)\eps^2$ (the role of $M$ is played by
  $N\eps$). That lemma tells us then that
  \begin{equation}\label{eq:6}
    \Leb(T_N\cap S_N) \gg \psi(N)\eps^3.
  \end{equation}
  Putting $U_N := T_N\cap S_N$, notice that we have for $M\neq N$
  \begin{align*}
    \Leb(U_M\cap U_N) &\leq \Leb(S_M\cap S_N) \\
                      &\leq 2 \Leb(S_M)\Leb(S_N) \\
                      &\leq 2 \parens*{2\psi(M)\eps^2}\parens*{2\psi(N)\eps^2} \\
    &\overset{\eqref{eq:6}}{\ll} \Leb(U_M)\Leb(U_N).
  \end{align*}
  The implicit constant in this last expression depends on $\eps$, but
  this is unimportant. What is important now is that $(U_N)_N$ is a
  sequence of subsets of $\TT^2$ which are pairwise quasi-independent,
  and have the property that
  \begin{equation*}
    \sum_{t=0}^\infty \Leb(U_{2^t}) \gg \sum_{t=0}^\infty \psi(2^t)
  \end{equation*}
  diverges. Therefore, $U_\infty :=\limsup_t U_{2^t}$ has full measure
  in $\TT^2$.

  We now construct an infinite set $\calA\subset\NN$ whose additive
  energy satisfies $E(A_N)\asymp N^3\psi(N)$, and such that for every
  $(\alpha, \gamma) \in U_\infty$, we have
  $\limsup_N F(\alpha, \gamma, 1, N, \calA) = \infty$. This will prove
  the theorem. 

  The set $\calA$ will consist of
  concatenated blocks
  $B_N$ of integers, each of which is a subset
  \begin{equation*}
   B_N \subset  \Delta_N\brackets*{(N\psi(N)^{-1}, 2N\psi(N)^{-1}]\cap\NN}.
 \end{equation*}
 In view of~\cite[Lemma~6]{Aistleitneretaladditiveenergy}, we may
 find, for each $N$, a block $B_N$ with the properties:
 \begin{enumerate}
 \item For all $d\in\ZZ\setminus\set{0}$ we have
   $r_{B_N}(\Delta_N d) \leq 2 N\psi(N)$.
 \item For all $d\in\ZZ\setminus\set{0}$ with
   $\abs{d} < \frac{N}{10\psi(N)}$ we have
   $r_{B_N}(\Delta_N d) \geq \frac{1}{2} N\psi(N)$.
  \item We have $N/2 \leq \# B_N \leq 2N$. 
  \end{enumerate}
  Using~\eqref{eq:9} and the first two properties above, we see that
  $E(B_N)\asymp N^3 \psi(N)$.

  Put $\calA = \set{B_1, B_2, B_4, \dots }$ as the concatenation of
  the blocks $B_{2^t}, t\geq 0$. Suppose that
  \begin{equation*}
\sum_{k=0}^{t-1} \# B_{2^k} < N \leq
  \sum_{k=0}^t \# B_{2^k},
\end{equation*}
that is, $A_N$ is a truncation of $\calA$ in the block
$B_{2^t}$. Clearly, we have $E(A_N) \geq E(B_{2^{t-1}})$, hence
  \begin{equation*}
    E(A_N) \gg (2^{t-1})^3\psi(2^{t-1}) \gg N^3\psi(N).
  \end{equation*}
  On the other hand, we may assume that the sequence $(\Delta_N)$ is
  sparse enough that 
  \begin{equation*}
    E(A_N) \leq \sum_{k=0}^{t} E(B_{2^k}),
  \end{equation*}
  that is, the only contributions to the additive energy come from
  four-tuples $(a,b,c,d)$ which lie in the same block. This leads to
  \begin{align}
    E(A_N) &\ll \sum_{k=0}^{t} (2^k)^3\psi(2^k) \nonumber \\
           &\ll (2^t)^3 \psi(2^t) \label{eq:13}\\
    &\ll N^3 \psi(N), \nonumber
  \end{align}
  where~\eqref{eq:13} follows from our assumption that
  $N^{3-\delta}\psi(N)$ is increasing.\footnote{This is the only place
    where that assumption is used.} Therefore, $E(A_N)\asymp
  N^3\psi(N)$, as needed.

  To estimate the pair correlations, note that for $N=2^t$, we have
\begin{align*}
  F(1, \#B_1+ \# B_2 + \# B_4+\dots+\# B_N, \calA) &\geq
                                                                  \frac{1}{4N}
                                                                  \sum_{d\neq
                                                                  0}
                                                                  r_{B_n}(\Delta_N
                                                                  d)\bone_{\Delta_N
                                                                  d,1/(4N)} \\
                                                                &\geq
                                                                  \frac{\psi(N)}{8}\sum_{0
                                                                  <
                                                                  \abs{d}\leq
                                                                  \frac{N}{10\psi(N)}}
                                                                  \bone_{\Delta_N
                                                                  d,1/(4N)}
  \\
                                                                &\geq
                                                                  \frac{\psi(N)}{8}\sum_{0
                                                                  <
                                                                  d\leq
                                                                  \frac{N}{10\psi(N)}}
                                                                  \bone_{\Delta_N
                                                                  d,1/(4N)}.
\end{align*}
Notice then that for any $\alpha\in U_N$, we will have 
\begin{equation*}
  F(\alpha, 1, \# B_1+\dots+\# B_N, \calA) \geq
  \frac{1}{160\eps}.
\end{equation*}
Since almost every $(\alpha,\gamma) \in\TT^2$ is contained in infinitely many
$U_{2^t}$'s, this implies that 
\begin{equation*}
\limsup_{N\to\infty} F(\alpha, \gamma, 1, N, \calA) \geq   \frac{1}{160
\eps}. 
\end{equation*}
for almost every $(\alpha,\gamma)\in\TT^2$.  If $0 < \eps < \frac{1}{320}$  then
this value exceeds $2$, therefore $\calA$ does not have
doubly metric Poissonian pair correlations and the theorem is proved.
\end{proof}

\section{Proofs of Theorems~\ref{thm:subsequenceofsubsequence} and~\ref{thm:subsequence}}

Theorem~\ref{thm:subsequenceofsubsequence} can be deduced from the
proof of Theorem~\ref{thm:fundyes}.

\begin{proof}[Proof of Theorem~\ref{thm:subsequenceofsubsequence}]
  By Lemma~\ref{lem:variance}, we have that
  \begin{equation*}
    \sigma^2(F(s,N)) \ll E(A)N^{-3} s.
  \end{equation*}
  Let $\set{N_t}\subset\NN$ be a sequence as in the theorem statement.
  By passing to a subsequence, we may assume that $\set{N_t}$
  increases fast enough that for any $s>0$, the sum
  $\sum_t \sigma^2(F(s,N_t))$ converges. Let $\eps,s>0$.  The first
  paragraph of the proof of Theorem~\ref{thm:fundyes}, with the sequence
  $\set{\floor{k^t}}$ replaced by the sequence $\set{N_t}$, shows that
  for almost every $(\alpha, \gamma)$,
  \begin{equation*}
    F(\alpha, \gamma, s, N_t, \calA) \to 2s
  \end{equation*}
  as $t\to \infty$. Running the argument for every rational $s>0$
  and intersecting the (countably many) full-probability subsets of
  phase space, we get that for almost every $(\alpha, \gamma)$,
  \begin{equation*}
    (\forall s\in\QQ_{>0})\qquad \lim_{t\to\infty} F(\alpha, \gamma, s,
    N_t, \calA) = 2s. 
  \end{equation*}
  Therefore,~\eqref{eq:dmppc} holds, as needed.
\end{proof}

Half of Theorem~\ref{thm:subsequence} is
Theorem~\ref{thm:subsequenceofsubsequence}. The other half follows
immediately from an inhomogeneous version of Bourgain's proof
in~\cite[Appendix]{Aistleitneretaladditiveenergy} that
$E(A) = \Omega(N^3)$ precludes MPPC. For completeness, we will carry
out the argument in the inhomogeneous setting, following an exposition
of Walker~\cite{WalkerSurvey}.

The basis of the argument is the Balog--Szem{\'e}redi--Gowers
Lemma. 

\begin{lemma}[{\cite[Section~2.5]{TaoVubook}}]\label{lem:bzg}
  Let $A\subset\ZZ$ be a finite set of integers. For any $c>0$ there
  exist $c_1, c_2>0$ depending only on $c$ such that the following
  holds. If $E(A)\geq c\#A^3$, then there is a subset $B\subset A$
  such that $\# B \geq c_1\#A$ and $\#\parens{B-B} \leq c_2\# A$.
\end{lemma}

\begin{theorem}[{\cite[Appendix]{Aistleitneretaladditiveenergy}},
  doubly metric inhomogeneous version]\label{thm:bourgain}
  Suppose that $\set{N_t}\subset \NN$ is a sequence such that we have
  $E(A_{N_t}) \geq c N_t^3$ for some constant $c>0$ and all large
  $t$. Then $\calA$ does not have doubly metric Poissonian pair
  correlations along $\set{N_t}$.
\end{theorem}

\begin{proof}
  For each $t\in \NN$ large enough, let $B_{N_t}$ be the subset of
  $A_{N_t}$ which is guaranteed by Lemma~\ref{lem:bzg}.

  Let $s>0$ be a fixed real number, to be specified. Let
  \begin{equation*}
    \Omega_t := \set*{(\alpha, \gamma)\in \TT^2 : \norm{\alpha n - \gamma} \leq \frac{s}{N_t}\textrm{ for some } n\in B_{N_t} - B_{N_t}},
  \end{equation*}
  and notice that
  \begin{equation*}
    \Leb\parens*{\Omega_t} \leq \frac{2s}{N_t} \#\parens*{B_{N_t} - B_{N_t}} \overset{\textrm{Lem.~\ref{lem:bzg}}}{\leq} 2s c_2.
  \end{equation*}
  Notice also that for every $(\alpha, \gamma)\in\TT^2\setminus\Omega_t$ we have
  \begin{equation*}
    F(\alpha, \gamma, s, N_t) = \frac{1}{N_t} \sum_{(a,b) \in A_{N_t}\times A_{N_t} \setminus B_{N_t}\times B_{N_t} } \bone_{\brackets*{0, s/N_t}} \parens*{\norm{\alpha(a-b) - \gamma}}, 
  \end{equation*}
  hence 
  \begin{align*}
    \int_{\TT^2\setminus \Omega_t} F(\alpha, \gamma, s, N_t)\,d\alpha\,d \gamma 
    &= \frac{1}{N_t}\sum_{(a,b) \in A_{N_t}\times A_{N_t} \setminus B_{N_t}\times B_{N_t} } \int_{\TT^2\setminus \Omega_t} \bone_{\brackets*{0, s/N_t}} \parens*{\norm{\alpha(a-b) - \gamma}}\\
    &\leq \frac{2s}{N_t^2}\#\parens*{A_{N_t}\times A_{N_t} \setminus B_{N_t}\times B_{N_t}} \\
    &\overset{\textrm{Lem.~\ref{lem:bzg}}}{\leq} 2s (1 - c_1^2). 
  \end{align*}
  Now, suppose that the set of
  $(\alpha, \gamma)\in \TT^2\setminus \Omega_t$ for which
  $F(\alpha, \gamma, s, N_t) \leq 2s (1 - c_1^2/4)$ has measure less
  than $c_1^2/4$. Then we would have 
  \begin{equation*}
    2s\parens*{1-c_1^2} \geq \int_{\TT^2\setminus\Omega_t} F(\alpha, \gamma, s, N_t)\,d\alpha\,d\gamma > 2s \parens*{1 - \frac{c_1^2}{4}}\parens*{1 - 2sc_2 - \frac{c_1^2}{4}}.
  \end{equation*}
  But for small enough $s>0$, this cannot possibly hold. Let us now
  specify $s>0$ to be small enough. Then there is a set
  $\Gamma_t \subset \TT^2\setminus\Omega_t$ with
  $\Leb(\Gamma_t) \geq c_1^2/4$ and such that
  $F(\alpha, \gamma, s, N_t) \leq 2s(1-c_1^2/4)$ holds for all
  $(\alpha, \gamma)\in\Gamma_t$. Let
  $\Gamma :=\limsup_{t\to\infty}\Gamma_t$, and notice that we must
  have $\Leb(\Gamma) > 0$. Since for any
  $(\alpha, \gamma)\in\Gamma$, we have that
  \begin{equation*}
    \liminf_{t\to\infty} F (\alpha, \gamma, s, N_t) \leq 2s \parens*{1 - \frac{c_1^2}{4}} < 2s, 
  \end{equation*}
  the theorem is proved.
\end{proof}

\begin{proof}[Proof of Theorem~\ref{thm:subsequence}]
  For the ``if'' direction note that if
  $\liminf N^{-3}E(A)=0$, then there is a subsequence
  $\set{N_t}\subset\NN$ to which
  Theorem~\ref{thm:subsequenceofsubsequence} applies.

  For the ``only if'' direction, suppose that
  $\liminf N^{-3}E(A)>0$, and let $\set{N_t}\subset\NN$ be any
  subsequence. Then there is some constant $c>0$ such that
  $E(A_{N_t}) \geq c N_t^{3}$ holds for all large $t$, and
  Theorem~\ref{thm:bourgain} implies that $\calA$ does not have DMPPC
  along $\set{N_t}$.
\end{proof}

\section{Proof of Theorem~\ref{eq:uid}}

The proof of Theorem~\ref{eq:uid} is adapted from~\cite{HKLSUPPC}. We
will use the following simple lemma stating that if a sequence does
not equidistribute in $\TT$, then there are arbitrarily small
intervals in $\TT$ which are ``overrepresented'' infinitely often.

\begin{lemma}\label{lem:proof-theor-refeq-3}
  Suppose $(x_n)_n$ is a sequence of points in $\TT$ which does not
  equidistribute. Then there are arbitrarily small intervals
  $I\subset\TT$ such that
  \begin{equation}\label{eq:2}
    \limsup_{N\to\infty}\frac{1}{N}\#\set*{n\in[N]\mid x_n\in
      I}>\Leb(I).
  \end{equation}
\end{lemma}

\begin{proof}
  For an interval $I\subset\TT$ and integer $N\in\NN$, let
  \begin{equation*}
    A_N(I) :=\frac{1}{N}\#\set*{n\in[N]\mid x_n\in I}.
  \end{equation*}
  That $(x_n)$ does not equidistribute implies that there are
  arbitrarily short intervals $I \subset\TT$ such that
  $\lim_{N\to\infty} A_N(I)\neq \Leb(I)$.  Let $I$ be such an
  interval, as small as desired. If~\eqref{eq:2} holds then we are
  done, so let us assume that
  $\limsup_{N\to\infty} A_N(I)\leq \Leb(I)$.  Therefore, it must be
  that $\liminf_{N\to\infty}A_N(I) < \Leb(I)$, which implies that
  there is some $\eps>0$ and infinitely many values of $N\in\NN$ for
  which $A_N(I)\leq \Leb(I) (1-\eps)$.

  Now, let $\bigcup_{k=0}^K I_k$ be a partition of $\TT$ by intervals
  where $I_0=I$ and $\Leb(I_k) = (1-\Leb (I))/K$ for each
  $k=1, \dots, K$. Notice that for all $N$ we have
  $\sum_{k=0}^K A_N(I_k) = 1$, so for the infinitely many $N$ for
  which $A_N(I)\leq \Leb(I) (1-\eps)$ holds we have
  \begin{equation*}
    \Leb(I)(1-\eps) + \sum_{k=1}^K A_N(I_k) \geq 1.
  \end{equation*}
  This implies that there must be some $\hat k := \hat k(N)\in [K]$
  for which
  \begin{align*}
    A_N(I_{\hat k}) &\geq \frac{1 - \Leb(I)(1-\eps)}{K} \\
                    &= \Leb(I_{\hat k}) \parens*{\frac{1 -
                      \Leb(I)(1-\eps)}{1 -\Leb(I)}}.
  \end{align*}
  By the pigeonhole principle, there is some $\bar k$ such that
  $\bar k = \hat k(N)$ for infinitely many $N$. For this $\bar k$, we
  have $\limsup_{N\to\infty}A_N(I_{\bar k}) > \Leb(I_{\bar k})$.
\end{proof}

\begin{proof}[Proof of Theorem~\ref{eq:uid}]
  Assume that the sequence $\mathbf x = (x_n)_n$ is not
  equidistributed in $\TT$. In view of the results
  of~\cite{AistleitnerLachmannPausinger,GrepstadLarcher,HKLSUPPC,Marklof2019},
  it follows that $\mathbf x$ does not have $\gamma$-PPC when
  $\gamma=0$. So let us fix arbitrarily $\gamma > 0$.

  The fact that the sequence is not equidistributed implies, by
  Lemma~\ref{lem:proof-theor-refeq-3}, that there is some arc in $\TT$
  which is ``underrepresented'' for infinitely many partial sequences
  $(x_n)_{n=1}^N$. Furthermore, that arc can be supposed to be as long
  as we like. In particular, since the properties of equidistribution
  and $\gamma$-PPC are not altered by rotating the entire set $\calA$,
  this proof will not lose any generality if we assume that there
  exist $\alpha, \beta>0$ such that $1- \alpha < \norm{\gamma}$ and
  such that for infinitely many $N\in \NN$ we have
  \begin{equation*}
    \frac{1}{N}\#\set*{1 \leq n \leq N \mid x_n \in [0,\alpha)} \leq \beta < \alpha. 
  \end{equation*}
  Let $N$ be one such (large) integer. For $i=0, \dots, N-1$, let
  \begin{equation*}
    X_i := \#\set*{1\leq n \leq N \mid x_n \in \Bigg\lbrack\frac{i}{N}, \frac{i+1}{N}\Bigg\rparen}
  \end{equation*}
  so that we have
  \begin{equation}\label{eq:constraints}
    \sum_{i=0}^{\floor{N\alpha}-1} X_i \leq N\beta \quad\textrm{while}\quad\sum_{i=0}^{N-1} X_i =  N.
  \end{equation}
  Now notice that for $s\in\NN$, 
  \begin{equation*}
    N F(\gamma, s,N, \mathbf x) \leq \sum_{i=0}^{N-1} \sum_{j  = -s}^s  X_i X_{\brackets*{i + \floor{\gamma N} + j \pmod N}}.
  \end{equation*}
  The right-hand side defines a quadratic form in the variables
  $X_0, \dots, X_{N-1}$ which, subject to the constraints imposed
  by~(\ref{eq:constraints}), reaches its maximum when
  \begin{equation}\label{eq:plugin}
    X_0 = X_1 = \dots = X_{\floor{N\alpha}-1} = \frac{N\beta}{\floor{N\alpha}} \quad\textrm{and}\quad X_{\floor{N\alpha}} = \dots = X_{N-1} = \frac{N(1-\beta)}{N-\floor{N\alpha}}.
  \end{equation}
  Now with a fixed $s\in\NN$ and $N$ large enough, we will have
  $1-\alpha < \norm{\gamma} - s/N$, hence, 
  \begin{align*}
    N F(\gamma, s,N, \mathbf x) &\leq \sum_{i=0}^{N-1} \sum_{j  = -s}^{s}  X_i X_{i + \floor{\gamma N} + j} \\
             &=\sum_{i=0}^{\floor{N\alpha}-1} X_i \sum_{j  = -s}^s X_{i + \floor{\gamma N} + j}  + \sum_{i=\floor{N\alpha}}^{N-1} X_i \sum_{j  = -s}^s X_{i + \floor{\gamma N} + j}  \\
             &\overset{\textrm{(\ref{eq:plugin})}}{\leq} \parens*{\frac{N\beta}{\floor{N\alpha}}}\parens*{\frac{N(1-\beta)}{N-\floor{N\alpha}}} \sum_{i=\floor{N\alpha}}^{N-1}\sum_{j=-s}^{s}2 + \parens*{\frac{N\beta}{\floor{N\alpha}}}^2\parens*{N(2s+1)- \sum_{i=\floor{N\alpha}}^{N-1}\sum_{j=-s}^s2} \\
             &\sim (2s+1)N\underbrace{\brackets*{\frac{\beta}{\alpha}\parens*{2-\frac{\beta}{\alpha}}}}_{<1}.
  \end{align*}
  We see that there exists $\theta<1$ depending only on $\alpha$ and
  $\beta$ with the property that for any $s\in\NN$ and infinitely many
  $N\in\NN$, we have $F(\gamma, s,N, \mathbf x) \leq (2s+1)\theta$.
  Therefore, if $s$ is large enough, it is impossible that
  $F(\gamma, s, N, \mathbf x) \to 2s$ as $N\to\infty$, so $\mathbf x$
  does not have $\gamma$-PPC.
\end{proof}

\subsection*{Acknowledgments}
\label{sec:acknowledgments}

I thank Christoph Aistleitner for an illuminating correspondence.

\bibliographystyle{plain}

%\bibliography{../bibliography}

\begin{thebibliography}{10}

\bibitem{AistleitnerLachmannPausinger}
Christoph Aistleitner, Thomas Lachmann, and Florian Pausinger.
\newblock Pair correlations and equidistribution.
\newblock {\em J. Number Theory}, 182:206--220, 2018.

\bibitem{AistleitnerLachmannTechnau}
Christoph Aistleitner, Thomas Lachmann, and Niclas Technau.
\newblock There is no {K}hintchine threshold for metric pair correlations.
\newblock {\em Mathematika}, 65(4):929--949, 2019.

\bibitem{Aistleitneretaladditiveenergy}
Christoph Aistleitner, Gerhard Larcher, and Mark Lewko.
\newblock Additive energy and the {H}ausdorff dimension of the exceptional set
  in metric pair correlation problems.
\newblock {\em Israel J. Math.}, 222(1):463--485, 2017.
\newblock With an appendix by Jean Bourgain.

\bibitem{Bloometal}
Thomas~F. Bloom, Sam Chow, Ayla Gafni, and Aled Walker.
\newblock Additive energy and the metric {P}oissonian property.
\newblock {\em Mathematika}, 64(3):679--700, 2018.

\bibitem{BloomWalker}
Thomas~F. Bloom and Aled Walker.
\newblock G{CD} sums and sum-product estimates.
\newblock {\em Israel J. Math.}, 235(1):1--11, 2020.

\bibitem{GrepstadLarcher}
Sigrid Grepstad and Gerhard Larcher.
\newblock On pair correlation and discrepancy.
\newblock {\em Arch. Math. (Basel)}, 109(2):143--149, 2017.

\bibitem{HKLSUPPC}
Aicke Hinrichs, Lisa Kaltenb\"{o}ck, Gerhard Larcher, Wolfgang Stockinger, and
  Mario Ullrich.
\newblock On a multi-dimensional {P}oissonian pair correlation concept and
  uniform distribution.
\newblock {\em Monatsh. Math.}, 190(2):333--352, 2019.

\bibitem{LachmannTechnau}
Thomas Lachmann and Niclas Technau.
\newblock On exceptional sets in the metric {P}oissonian pair correlations
  problem.
\newblock {\em Monatsh. Math.}, 189(1):137--156, 2019.

\bibitem{LarcherStockinger}
Gerhard Larcher and Wolfgang Stockinger.
\newblock On pair correlation of sequences.
\newblock Survey, 2019.

\bibitem{LarcherStockingerMaxE}
Gerhard Larcher and Wolfgang Stockinger.
\newblock Pair correlation of sequences with maximal additive energy.
\newblock {\em Math. Proc. Cambridge Philos. Soc.}, 168(2):287--293, 2020.

\bibitem{LarcherStockingerNegResults}
Gerhard Larcher and Wolfgang Stockinger.
\newblock Some negative results related to {P}oissonian pair correlation
  problems.
\newblock {\em Discrete Math.}, 343(2):111656, 11, 2020.

\bibitem{MarklofDistmodone}
Jens Marklof.
\newblock Distribution modulo one and {R}atner's theorem.
\newblock In {\em Equidistribution in number theory, an introduction}, volume
  237 of {\em NATO Sci. Ser. II Math. Phys. Chem.}, pages 217--244. Springer,
  Dordrecht, 2007.

\bibitem{Marklof2019}
Jens Marklof.
\newblock Pair correlation and equidistribution on manifolds.
\newblock {\em Monatshefte f{\"u}r Mathematik}, Jun 2019.

\bibitem{RudnickSarnak}
Ze\'{e}v Rudnick and Peter Sarnak.
\newblock The pair correlation function of fractional parts of polynomials.
\newblock {\em Comm. Math. Phys.}, 194(1):61--70, 1998.

\bibitem{TaoVubook}
Terence Tao and Van Vu.
\newblock {\em Additive combinatorics}, volume 105 of {\em Cambridge Studies in
  Advanced Mathematics}.
\newblock Cambridge University Press, Cambridge, 2006.

\bibitem{Walker}
Aled Walker.
\newblock The primes are not metric {P}oissonian.
\newblock {\em Mathematika}, 64(1):230--236, 2018.

\bibitem{WalkerSurvey}
Aled Walker.
\newblock Additive combinatorics: some new techniques for pair correlation
  problems.
\newblock Notes for a minicourse, 2019.

\end{thebibliography}

\end{document}